\documentclass[letterpaper, 11pt]{article}
\usepackage[margin = 1in]{geometry}
\usepackage{graphicx,color}
\usepackage[cp850]{inputenc}
\usepackage[T1]{fontenc}
\usepackage[english]{babel}
\usepackage{cite}
\usepackage{rotating}
\usepackage[f]{esvect}
\usepackage{amsmath, amsthm, amssymb}
\usepackage{rotating}
\usepackage{lscape}
\usepackage{mathrsfs}
\usepackage{url}
\def\tablenotes{\bgroup\parfillskip=0pt plus 1fil
\leftskip=0pt\relax \rightskip=0pt
\vskip2pt\footnotesize}
\def\endtablenotes{\vskip1pt\egroup}

\newtheorem{theorem}{Theorem}[section]

\newtheorem{proposition}[theorem]{Proposition}
\theoremstyle{definition}

\begin{document}

\title{A recursive algorithm for an efficient and accurate computation of incomplete Bessel functions}
\author{Richard M. Slevinsky$^\dag$ and Hassan Safouhi$^{\S,}$\footnote{The corresponding author (HS) acknowledges the financial support from the Natural Sciences and Engineering Research Council of Canada~(NSERC) - Grant RGPIN-2016-04317}\\
\\
$^\dag${\it S2.29 Mathematical Institute}\\
{\it University of Oxford}\\
{\it Andrew Wiles Building, Radcliffe Observatory Quarter}\\
{\it Woodstock Road, Oxford UK OX2 6GG}\\
\\
$^{\S}$ {\it Mathematical Section}\\
{\it Campus Saint-Jean, University of Alberta}\\
{\it 8406, 91 Street, Edmonton (AB), Canada T6C 4G9}}

\date{} 
\maketitle

\vspace*{1.0cm}

\underline{\bf Abstract} \hskip 0.15cm In a previous work, we developed an algorithm for the computation of incomplete Bessel functions, which pose as a numerical challenge, based on the $G_{n}^{(1)}$ transformation and Slevinsky-Safouhi formula for differentiation. In the present contribution, we improve this existing algorithm for incomplete Bessel functions by developing a recurrence relation for the numerator sequence and the denominator sequence whose ratio forms the sequence of approximations. By finding this recurrence relation, we reduce the complexity from ${\cal O}(n^4)$ to ${\cal O}(n)$. We plot relative error showing that the algorithm is capable of extremely high accuracy for incomplete Bessel functions.


\maketitle

\clearpage
\section{Introduction}\label{sec1}

Incomplete Bessel functions, which are a computational challenge, were a subject of significant research. These functions appear when describing a plethora of phenomena in hydrology, statistics, and quantum mechanics~\cite{Jones-50-173-07, Jones-50-711-07, Hantush-Jacob-36-95-55, Harris-63-913-97, Harris-70-623-98, Harris-81-332-01, Hunt-33-179-77, Kryachko-78-303-00, Harris-Fripiat-109-1728-09}. Incomplete Bessel functions of zero order are also involved in a numerous applications to electromagnetic waves~\cite{Agrest-Maksimov-71, Lewin-AP19-134-71, Chang-Fisher-9-1129-74, Dvorak-36-26-94, Mechaik-Dvorak-8-1563-94}. By introducing a recursive algorithm for the $G$ transformation for tail integrals and applying it to incomplete Bessel functions, the article~\cite{Safouhi40} has received some attention recently~\cite{Safouhi43, Safouhi40b, Safouhi47, Creal-12, Creal-13, Nestler-Pippig-Potts-13}. Most of this attention has been due to the algorithm for incomplete Bessel functions.

Integral representations of the incomplete Bessel functions are given by~\cite{Harris-215-260-08}:
\begin{equation}
K_\nu(x,y) = \int_1^\infty \frac{e^{-x \,t - y/t}}{t^{\nu+1}} {\rm\,d} t,
\label{eq:inc1}
\end{equation}

The algorithm for the $G$ transformation~\cite{Gray-Wang-29-271-92, Gray-Atchison-4-363-67, Gray-tchison-McWilliams-8-365-71, Joyce-13-435-71} applied to incomplete Bessel functions~\cite{Safouhi40} appears as:
\begin{equation}
\tilde{G}^{(1)}_n(x,y,\nu) = x^\nu\frac{\tilde{\cal N}_n(x,y,\nu)}{{\cal D}_n(x,y,\nu)},
\label{EQGNXY}
\end{equation}
where:
\begin{equation}
{\cal D}_n(x,y,\nu) = (-x\,y)^n\,x^{\nu+1}\,e^{x+y} \sum_{r=0}^n \binom{n}{r} (-y)^{-r} \sum_{i=0}^r A_r^i\,x^i,
\end{equation}
where the $A_r^i$ are the coefficients in the Slevinsky-Safouhi formula I~\cite{Safouhi37} and are given by:
\begin{equation}\label{eq:theoremcoeff}
A^i_k = \left\{ \begin{array}{lll}
1 & \quad \textrm{for} \quad & i = k,\\
\displaystyle (n-\nu-(k-1)(\mu+1))A_{k-1}^0 & \quad \textrm{for} \quad & i = 0,~k>0, \\
(n-\nu+i(m+1)-(k-1)(\mu+1))A^i_{k-1}+A^{i-1}_{k-1} & \quad \textrm{for} \quad &  0<i<k,
\end{array}
\right.
\end{equation}
with $(\mu,\nu,m,n) = (-2,-\nu-1,0,0)$.

The numerator $\tilde{\cal N}_n(x,y,\nu)$ in~\eqref{EQGNXY} is given by:
\begin{equation}
\tilde{\cal N}_n(x,y,\nu) = \frac{e^{-x-y}}{x^\nu\,y}\sum_{r=1}^n\binom{n}{r} \, {\cal D}_{n-r}(x,y,\nu)\,(x\,y)^r\sum_{s=0}^{r-1}\binom{r-1}{s}\,y^{-s}\sum_{i=0}^sA_s^i(-x)^i,
\end{equation}
where the $A_s^i$ are the coefficients in the Slevinsky-Safouhi formula I with $(\mu,\nu,m,n) = (-2,\nu-1,0,0)$.

Therefore, the original algorithm involves sums over sums over sums, effectively making the complexity of the computation of the numerator an ${\cal O}(n^3)$ process and the complexity of the computation of the denominator an ${\cal O}(n^2)$ process. Therefore, calculating a sequence of approximations $\left\{G_n\right\}_{n\ge0}$ results in an ${\cal O}(n^4)$ algorithm.

In the following, we introduce an algorithm for the $G$ transformation that reduce the calculation to an easily programmable and parallel four-term recurrence relation: with one initialization, it computes the numerator; with another initialization, it computes the denominator. As there are only a finite number of arithmetic operations required in the recurrence relation, the resulting algorithm for calculating a sequence of approximations $\left\{G_n\right\}_{n\ge0}$ is an ${\cal O}(n)$ algorithm.

The inductive proof of this recurrence relation follows and the notation is heavy because of the two variables $x$ and $y$ and the parameter $\nu$ already included in the incomplete Bessel function. As well, with this new algorithm comes the ability to program the $G$ transformation to unprecedentedly high order. We show some numerical results in the form of figures of the relative error for six different values. The reduction in complexity of the original algorithm comes with the benefit of a stable algorithm for high order.

\section{The algorithm}\label{sec2}

\begin{theorem}~\cite{Levin-Sidi-9-175-81}\label{thm:levinsidi} Let $f(x)$ be integrable on $[0,\infty)$ (i.e. $\int_0^{\infty}f(t)\, \mathrm{d} t \,<\, \infty$) and satisfy a linear differential equation of order $m$ of the form:
\begin{equation}
f(x) = \displaystyle \sum_{k=1}^m  p_k(x)\, f^{(k)}(x),
\end{equation}
where $p_k$ for $k=1,2,\ldots,m$ have asymptotic expansions as $x \to \infty$, of the form:
\begin{eqnarray}
p(x) & \sim & x^{i_k} \sum_{i=0}^\infty \frac{a_i}{x^i} \quad \textrm{with} \quad i_k \, \le \, k.
\label{eq:asympfirst}
\end{eqnarray}

If for $1 \le i \le m$ and $i \le k \le m$, we have:
\begin{equation}
\displaystyle \lim_{x\to\infty} p_k^{(i-1)}(x) \, f^{(k-i)}(x) = 0,
\end{equation}
and for every integer $l \ge -1$, we have:
\begin{equation}
\displaystyle \sum_{k=1}^m l (l-1) \cdots(l-k+1) p_{k,0} \ne 1,
\end{equation}
where $p_{k,0} = \displaystyle \lim_{x\to\infty} x^{-k}\, p_k(x)$ for $1 \le k \le m$, then as $x \to \infty$, we have:
\begin{equation}\label{eq:asympsecond}
\int_x^\infty f(t)\, \mathrm{d} t \sim \sum_{k=0}^{m-1} x^{j_k}\, f^{(k)}(x) \left(\beta_{0,k} + \frac{\beta_{1,k}}{x} + \frac{\beta_{2,k}}{x^2}+\cdots+ \right),
\end{equation}
where $j_k \le \mathrm{max}(i_{k+1}, i_{k+2}-1, \ldots, i_{m}-m+k+1) \quad for \quad k=0,1,\ldots,m-1$.
\end{theorem}

To solve for the unknowns $\beta_{k,i}$, we must set up and solve a system of linear equations.

The $G_{n}^{(m)}$ transformation~\cite{Gray-Wang-29-271-92} truncates the asymptotic expansions~\eqref{eq:asympsecond} after $n$ terms each and the system is formed by differentiation. The approximation $G_{n}^{(m)}$ to $\int_0^\infty f(t)\mathrm{d} t$ is given as the solution of the system of $mn+1$ linear equations~\cite{Gray-Wang-29-271-92}:
\begin{equation}\label{eq:gtrans}
\frac{\mathrm{d}^l}{\mathrm{d} x^l}\left\{G^{(m)}_n - \int_0^{x}f(t)\, \mathrm{d} t - \sum_{k=0}^{m-1}x^{\sigma_k}f^{(k)}(x)\sum_{i=0}^{n-1}\frac{\bar{\beta}_{k,i}}{x^i}\right\} = 0,\quad l=0,1,\ldots, mn,
\end{equation}
where it is assumed that $\displaystyle\frac{\mathrm{d}^l}{\mathrm{d} x^l}G^{(m)}_n\equiv0,\forall l>0$.
In the above system~\eqref{eq:gtrans}, $\sigma_k = \mathrm{min}(s_k,k+1)$ where $s_k$ is the largest of the integers $s$ such that $\displaystyle \lim_{x\to\infty}x^sf^{(k)}(x)=0$ holds, $k=0,1,\ldots,m-1$. Also, $G^{(m)}_n$ and $\bar{\beta}_{k,i}$ are the respective set of $mn+1$ unknowns.

The $G^{(1)}_{n}$ transformation can be written as the solution to the linear system~\eqref{eq:gtrans} with $m=1$. Instead of solving the system of linear equations each time for each order $n$, it would be ideal to resolve each approximation $G^{(1)}_n$ in a recursive manner~\cite{Safouhi40}.

By considering the equation~\eqref{eq:gtrans} for $l=0$:
\begin{equation}
G^{(1)}_n - F(x) = x^{\sigma_0}f(x)\sum_{i=0}^{n-1}\frac{\bar{\beta}_{0,i}}{x^i} \quad \textrm{with} \quad F(x) \,=\, \int_0^{x}f(t)\mathrm{d} t,
\end{equation}
and by isolating the summation on the right hand side, we obtain:
\begin{equation}
\displaystyle\frac{\displaystyle G^{(1)}_n - F(x)}{x^{\sigma_0}f(x)} = \sum_{i=0}^{n-1}\frac{\bar{\beta}_{0,i}}{x^i}.
\end{equation}

To eliminate the summation, and consequently all of the unknowns $\bar{\beta}_{0,i}$, we apply the $\displaystyle \left(x^2 \frac{\mathrm{d}}{\mathrm{d} x}\right)$ operator $n$ times, obtaining:
\begin{equation}
\left(x^2 \frac{\mathrm{d}}{\mathrm{d} x}\right)^n \left[ \frac{\displaystyle G^{(1)}_n - F(x)}{x^{\sigma_0}f(x)}\right] = 0 \quad \Longrightarrow \quad G^{(1)}_n = \displaystyle \frac{\left(x^2 \frac{\mathrm{d}}{\mathrm{d} x}\right)^n  \left(\frac{\displaystyle F(x)}{x^{\sigma_0}f(x)}\right)} {\left(x^2 \frac{\mathrm{d}}{\mathrm{d} x}\right)^n  \left(\frac{1}{x^{\sigma_0}f(x)}\right)},
\end{equation}
which leads to a recursive algorithm for the $G^{(1)}_n$ transformation.

\begin{enumerate}
\item Set:
\begin{equation}
{\cal N}_0(x) = \frac{F(x)}{x^{\sigma_0} f(x)} \qquad \textrm{and} \qquad {\cal D}_0(x) = \frac{1}{x^{\sigma_0} f(x)}.
\end{equation}

\item For $n=1,2,\ldots,$ compute ${\cal N}_n(x)$ and ${\cal D}_n(x)$ recursively from:
\begin{equation}
{\cal N}_{n}(x) = \left(x^2 \frac{\mathrm{d}}{\mathrm{d} x}\right) {\cal N}_{n-1}(x) \qquad \textrm{and} \qquad {\cal D}_{n}(x) = \left(x^2 \frac{\mathrm{d}}{\mathrm{d} x}\right) {\cal D}_{n-1}(x).
\label{eq:numden}
\end{equation}

\item For all $n$, the approximations $G^{(1)}_n(x)$ to $\displaystyle\left(\int_0^x+\int_x^\infty\right) f(t)\, \mathrm{d} t$ are given by:
\begin{equation}
G^{(1)}_n(x) = \frac{{\cal N}_n(x)}{{\cal D}_n(x)}.
\end{equation}
\end{enumerate}

For the incomplete Bessel functions, $\sigma_0=0$, and the algorithm for the $G_n^{(1)}$ transformation takes the form below. Let:
\begin{equation}
f(t) = \dfrac{e^{-xt-y/t}}{t^{\nu+1}} \quad \textrm{and} \quad \displaystyle F(t) = \int_0^tf(s){\rm\,d}s.
\end{equation}
\begin{enumerate}
\item Set:
\begin{equation}
{\cal N}_0(x,y,\nu,t) = \frac{F(t)}{f(t)} \qquad \textrm{and} \qquad {\cal D}_0(x,y,\nu,t) = \frac{1}{f(t)}.
\end{equation}

\item For $n=1,2,\ldots,$ compute ${\cal N}_n(x,y,\nu,t)$ and ${\cal D}_n(x,y,\nu,t)$ recursively from:
\begin{eqnarray}
{\cal N}_{n}(x,y,\nu,t) & = & \left(t^2 \frac{\mathrm{d}}{\mathrm{d} t}\right) {\cal N}_{n-1}(x,y,\nu,t)
\nonumber\\
{\cal D}_{n}(x,y,\nu,t) & = & \left(t^2 \frac{\mathrm{d}}{\mathrm{d} t}\right) {\cal D}_{n-1}(x,y,\nu,t).
\end{eqnarray}

\item For all $n\ge0$, the approximations $\tilde{G}^{(1)}_n(x,y,\nu)$ to $K_\nu(x,y)$ are given by:
\begin{equation}
\tilde{G}^{(1)}_n(x,y,\nu) = \frac{\tilde{\cal N}_n(x,y,\nu,1)}{{\cal D}_n(x,y,\nu,1)},
\end{equation}
where:
\begin{equation}
\tilde{\cal N}_n(x,y,\nu,t) = {\cal N}_n(x,y,\nu,t)-F(t){\cal D}_n(x,y,\nu,t).
\end{equation}
\end{enumerate}

\begin{proposition}\label{alg:recrel}
Let:
\begin{align}
\tilde{N}_0(x,y,\nu) = 0 & \qquad \textrm{and} \qquad  \tilde{N}_1(x,y,\nu) = 1,\label{eq:Ninit}\\
D_0(x,y,\nu) = e^{x+y} & \qquad \textrm{and} \qquad  D_1(x,y,\nu) = (x+\nu+1-y)D_0(x,y,\nu),\label{eq:Dinit}
\end{align}
and:
\begin{equation}
\tilde{N}_{-1}(x,y,\nu) = D_{-1}(x,y,\nu)=0.
\end{equation}

Let also:
\begin{eqnarray}
(n+1)Q_{n+1}(x,y,\nu) & = & (x+\nu+1+2n-y) \, Q_n(x,y,\nu)
\nonumber\\ & + & (2y-\nu-n) \, Q_{n-1}(x,y,\nu) - y \, Q_{n-2}(x,y,\nu),
\label{eq:Qrec}
\end{eqnarray}
where the $Q_n(x,y,\nu)$ stand for either the $\tilde{N}_n(x,y,\nu)$ or the $D_n(x,y,\nu)$.

Then:
\begin{equation}
\tilde{G}_{n}^{(1)}(x,y,\nu)  = \dfrac{\tilde{N}_n(x,y,\nu)}{D_n(x,y,\nu)}.\label{eq:Grep}
\end{equation}
\end{proposition}

\begin{proof}[\bf Proof.]
We begin the proof with the denominators.

Let ${\cal D}_n(t) = {\cal D}_n(x,y,\nu,t)$ for short and let ${\cal D}_{-2}(t) = {\cal D}_{-1}(t) = 0$.

The sequence $\{{\cal D}_n(t)\}_{n\ge0}$ is generated by the one-term recurrence:
\begin{equation}
{\cal D}_n(t) = \left(t^2\dfrac{\rm d}{{\rm d} t}\right){\cal D}_{n-1}(t).
\end{equation}

Then we show that ${\cal D}_n(t)$ satisfies, for $n\ge0$:
\begin{eqnarray}\label{eq:Dinproofrec}
{\cal D}_{n+1}(t) & = & (xt^2+(\nu+1+2n)t-y) {\cal D}_n(t)
\nonumber\\ & + & (2nty-n(n-1)t^2-n(\nu+1)t^2) D_{n-1}(t) - n(n-1)t^2 y D_{n-2}(t).
\end{eqnarray}

For $n=0$:
\begin{eqnarray}
{\cal D}_1(t) & = & \left(t^2\dfrac{\rm d}{{\rm d} t}\right){\cal D}_0(t)
\nonumber\\ & = & \left(t^2\dfrac{\rm d}{{\rm d} t}\right)t^{\nu+1}e^{xt+y/t}
\nonumber\\ & = & (xt^2+(\nu+1)t-y){\cal D}_0(t).
\end{eqnarray}

The induction argument assumes:
\begin{eqnarray}\label{eq:Dindarg}
{\cal D}_{k+1}(t) & = & (xt^2+(\nu+1+2k)t-y) {\cal D}_k(t)
\nonumber\\ & + & (2kty-k(k-1)t^2 - k(\nu+1)t^2) D_{k-1}(t)-k(k-1) t^2 y D_{k-2}(t).
\end{eqnarray}

Differentiating and multiplying by $t^2$, we obtain:
\begin{align}
{\cal D}_{k+2}(t) & = (xt^2+(\nu+1+2k)t-y){\cal D}_{k+1}(t)
\nonumber\\& + (2xt^3+(\nu+1+2k)t^2+2kty-k(k-1)t^2-k(\nu+1)t^2)D_{k}(t)
\nonumber\\& + (2kt^2y-2k(k-1)t^3-2k(\nu+1)t^3-k(k-1)t^2y)D_{k-1}(t)
\nonumber\\& -  2k(k-1)t^3yD_{k-2}(t).
\label{eq:Dindarg2}
\end{align}

Multiplying~\eqref{eq:Dindarg} by $2t$ and subtracting it from~\eqref{eq:Dindarg2}, we obtain, after simplification:
\begin{align}
{\cal D}_{k+2}(t) & = (xt^2+(\nu+1+2(k+1))t-y){\cal D}_{k+1}(t)
\nonumber\\& + (2(k+1)ty-k(k+1)t^2-(k+1)(\nu+1)t^2)D_{k}(t)-k(k+1)t^2yD_{k-1}(t),
\end{align}
which proves~\eqref{eq:Dinproofrec} by induction.

As with the denominator, let ${\cal N}_n(t) = {\cal N}_n(x,y,\nu,t)$ and $\tilde{\cal N}_n(t) = \tilde{\cal N}_n(x,y,\nu,t)$ for short.

Since ${\cal N}_0(t) = F(t){\cal D}_0(t)$:
\begin{eqnarray}
{\cal N}_1(t) & = & F(t){\cal D}_1(t) + t^2f(t){\cal D}_0(t)
\nonumber\\ & = & F(t){\cal D}_1(t) + t^2.
\end{eqnarray}

But we now write $F(t) = \dfrac{{\cal N}_0(t)}{{\cal D}_0(t)}$ to conclude that:
\begin{equation}\label{eq:Nindarg}
{\cal N}_1(t) = (t^2x+(\nu+1)t-y){\cal N}_0(t) + t^2.
\end{equation}

Differentiating and multiplying by $t^2$, we obtain:
\begin{equation}\label{eq:Nindarg2}
{\cal N}_2(t) = (t^2x+(\nu+1)t-y){\cal N}_1(t)+(2xt^3+(\nu+1)t^2){\cal N}_0(t) + 2t^3.
\end{equation}

Multiplying~\eqref{eq:Nindarg} by $2t$ and subtracting it from~\eqref{eq:Nindarg2}, we obtain, after simplification:
\begin{equation}
{\cal N}_2(t) = (t^2x+(\nu+3)t-y){\cal N}_1(t)+(2yt-(\nu+1)t^2){\cal N}_0(t).
\end{equation}

But this is just~\eqref{eq:Dinproofrec} for $n=1$ with the labels ${\cal N}_n(t)$ interchanged for ${\cal D}_n(t)$. Any further differentiation and multiplication by $t^2$ will, therefore, ultimately lead to the same four-term recurrence relation, the difference being different initial conditions.

As a sequences:
\begin{equation}
\{\tilde{\cal N}_n(t)\}_{n\ge0} = \{{\cal N}_n(t) - F(t){\cal D}_n(t)\}_{n\ge0},
\end{equation}
is a linear combination of both solutions ${\cal N}_n(t)$ and ${\cal D}_n(t)$.

Therefore, $\tilde{\cal N}_n(t)$ satisfies~\eqref{eq:Dinproofrec} as well with the labels appropriately interchanged.

To complete the proof, we must return to the original sequences:
\begin{equation}
\{\tilde{N}_n(x,y,\nu)\}_{n\ge0} \qquad \textrm{and} \qquad \{D_n(x,y,\nu)\}_{n\ge0}.
\end{equation}

Indeed, the relationship is that:
\begin{equation}
Q_n(x,y,\nu) = \dfrac{{\cal Q}_n(x,y,\nu,1)}{n!},
\end{equation}
where the $Q_n(x,y,\nu)$ stand for either the $\tilde{N}_n(x,y,\nu)$ or the $D_n(x,y,\nu)$ and the ${\cal Q}_n(x,y,\nu,t)$ stand for either the $\tilde{\cal N}_n(x,y,\nu,t)$ or the ${\cal D}_n(x,y,\nu,t)$.
\end{proof}

\section{Figures}\label{sec5}

In Figures~\ref{fig:IBFSplots1} and \ref{fig:IBFSplots2}, the relative error of the $G$ transformation is shown for several different values of $x$, $y$, and $\nu$. This figure shows the excellent stability and convergence properties of the improved algorithm that come with the use of the stable four-term recurrence relation.
\begin{figure}[!ht]
\centering
\includegraphics[width=0.80\textwidth]{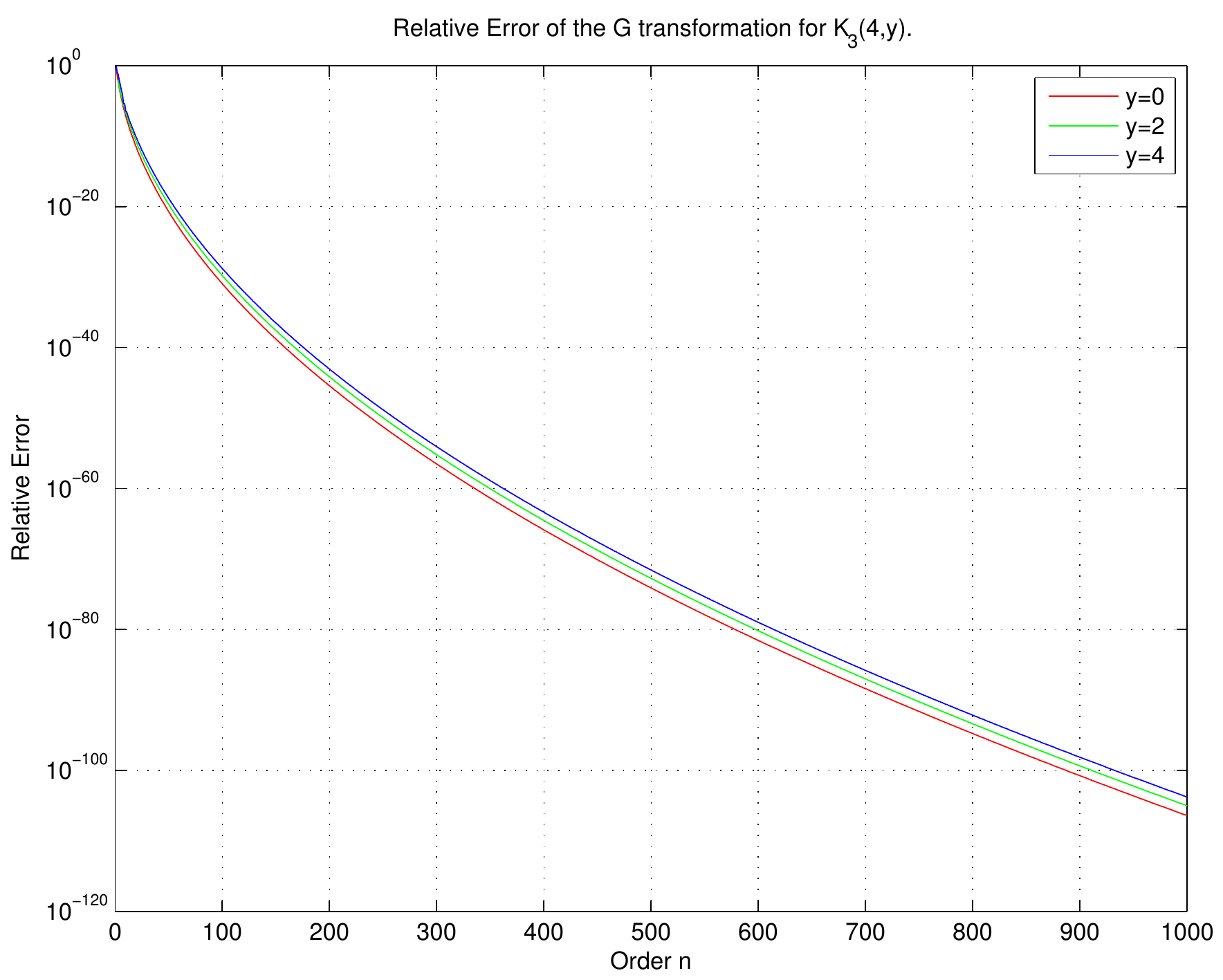}
\caption{Plots of the relative error of the $G$ transformation for $K_3(4,y)$, $y=0,2,4$.}
\label{fig:IBFSplots1}
\end{figure}

\clearpage

\begin{figure}[!ht]
\centering
\includegraphics[width=0.80\textwidth]{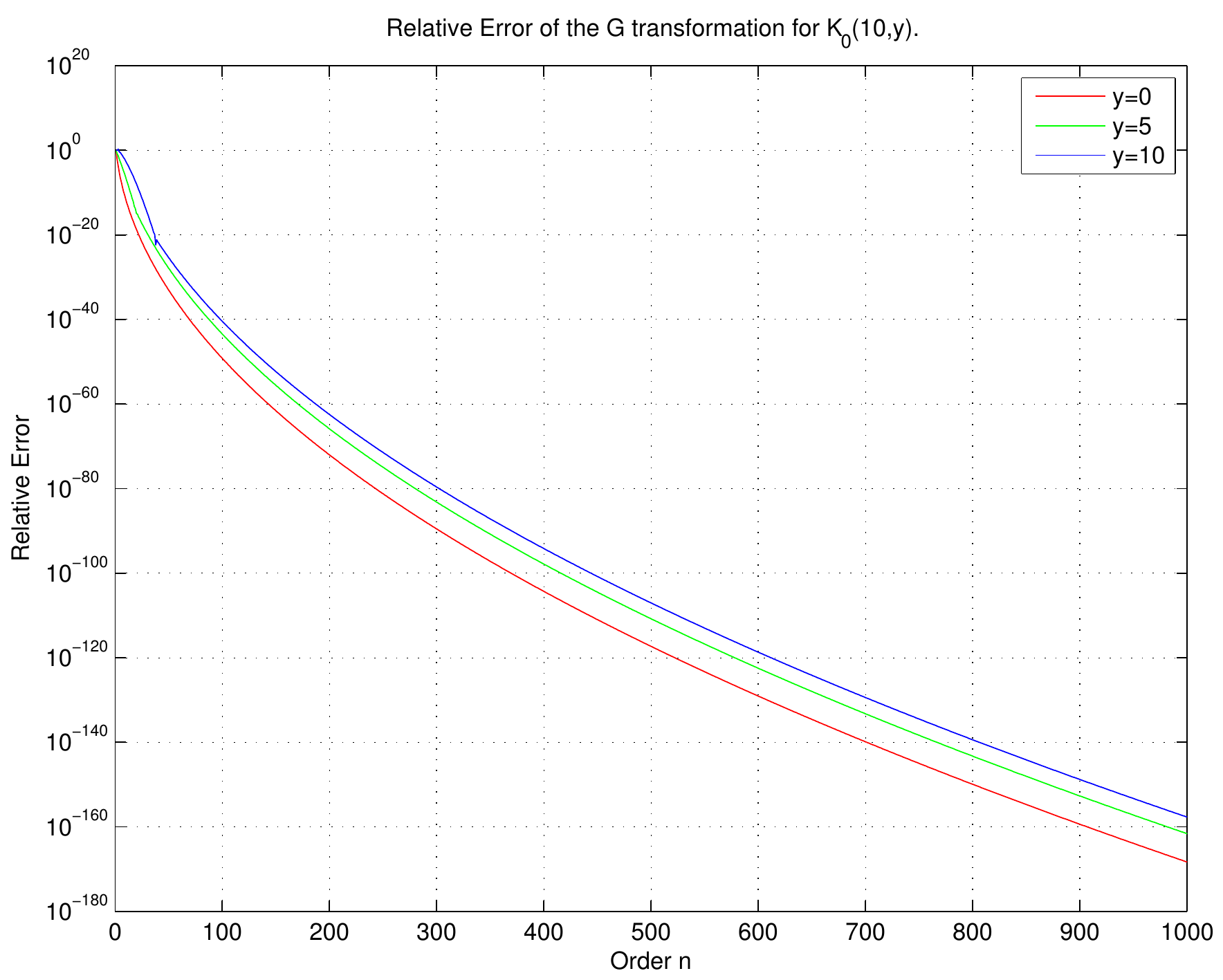}\\
\caption{Plots of the relative error of the $G$ transformation for $K_0(10,y)$, $y=0,5,10$.}
\label{fig:IBFSplots2}
\end{figure}

\section{Conclusion}\label{sec4}

We improve an existing algorithm for incomplete Bessel functions based on the $G$ transformation~\cite{Safouhi40} by developing a recurrence relation the numerator sequence and the denominator sequence whose ratio form the sequence of approximations. By finding this recurrence relation, we reduce the complexity from ${\cal O}(n^4)$ to ${\cal O}(n)$. We plot relative error showing that the algorithm is capable of extremely high accuracy for incomplete Bessel functions. The stability and convergence appear to be remarkable.

\end{document}